\documentclass[reqno,a4paper,9pt]{amsart}
\usepackage{marginnote}
\usepackage[utf8]{inputenc} % allow utf-8 input
\usepackage[english]{babel}
\usepackage[T1]{fontenc}
\usepackage{lmodern}
\usepackage{caption}
\usepackage{subcaption}
\usepackage{enumerate}
\usepackage{amssymb, amsmath}
\usepackage{mathrsfs,dsfont}
\usepackage{amscd}
\usepackage{bbm}
\usepackage{amsthm}
\usepackage{ esint }
\usepackage[mathcal]{euscript}
\usepackage[active]{srcltx}
\usepackage{verbatim}
\usepackage[colorlinks,linkcolor={blue},citecolor={blue},urlcolor={black}]{hyperref}
\usepackage[]{changebar}
\usepackage{xcolor}

\usepackage{mathtools}
\usepackage{hyperref}
\usepackage{url}            % simple URL typesetting
\usepackage{booktabs}       % professional-quality tables
\usepackage{amsfonts}       % blackboard math symbols
\usepackage{nicefrac}       % compact symbols for 1/2, etc.
\usepackage{microtype}      % microtypography
\usepackage{lipsum}
\usepackage{fancyhdr}       % header
\usepackage{graphicx}       % graphics
\graphicspath{{media/}}     % organize your images and other figures under media/ folder

\usepackage[numbers]{natbib}
\theoremstyle{plain}
\newtheorem{theorem}{Theorem}[section]

\newtheorem{lemma}[theorem]{Lemma}
\newtheorem{proposition}[theorem]{Proposition}

\newtheorem{definition}[theorem]{Definition}

\newtheorem*{definition*}{Definition}

\theoremstyle{remark}
\newtheorem{remark}[theorem]{Remark}
\newtheorem{example}[theorem]{Example}
\newtheorem*{remark*}{Remark}
\newtheorem*{example*}{Example}
\newtheorem*{notation*}{Notation}

\numberwithin{equation}{section}

\def\N{{\mathbb N}}

\def\R{{\mathbb R}}

\def\S{{\mathbb S}}

\def\norma #1{ \left\vert\left\vert #1 \right\vert\right\vert}
\def\nm #1{ \left\langle #1 \right\rangle}

\def\P{{\mathbb{P}}}
\def\AB{{\rm AB}}

\def\H{{\mathscr H}}
\def\BV{{\rm BV}}

\def\T{{\mathcal T}}

\def\L{\mathcal{L}}

\def\B{\mathscr{B}}
\def\X{\mathscr{X}}

\def\PR{{\rm Per}}
\def\w2{\stackrel{2}{\rightharpoonup}}
\def\weak{\rightharpoonup}
\def\strong2{\stackrel{2}{\rightarrow}}

\definecolor{viola}{rgb}{0.3,0,0.7}
\definecolor{ciclamino}{rgb}{0.5,0,0.5}
\definecolor{rosso}{rgb}{0.8,0,0}

\renewcommand{\P}{{\mathbb P}}
\newcommand{\F}{{\mathcal F}}

\newcommand{\eps}{\varepsilon}
\newcommand{\beq}{\begin{equation}}
\newcommand{\eeq}{\end{equation}}
\newcommand{\bal}{\begin{aligned}}
\newcommand{\eal}{\end{aligned}}
\newcommand{\ben}{\begin{enumerate}}
\newcommand{\beni} {\begin{enumerate}[(i)]}
\newcommand{\een}{\end{enumerate}}
\newcommand{\bit}{\begin{itemize}}
\newcommand{\eit}{\end{itemize}}
\newcommand{\beqw}{\begin{equation*}}
\newcommand{\eeqw}{\end{equation*}}
\newcommand{\bex}{\begin{example}}
\newcommand{\eex}{\end{example}}
\newcommand{\bre}{\begin{example}}
\newcommand{\ere}{\end{example}}
\newcommand{\bma}{\begin{bmatrix}}
\newcommand{\ema}{\end{bmatrix}}

\newcommand{\cH}{\mathcal{H}}

\newcommand{\cE}{\mathcal{E}}
\newcommand{\cA}{\mathcal{A}}

\newcommand{\de}{\mathrm{d}}
\renewcommand{\div}{\operatorname{div}}

\newcommand{\G}{\mathscr{G}}

\newcommand{\cP}{\mathscr{P}}

\renewcommand{\tilde}{\widetilde}

\newcommand{\res}{\mathop{\hbox{\vrule height 7pt width .5pt depth
0pt\vrule height .5pt width 6pt depth0pt}}\nolimits}

\makeatletter
\@namedef{subjclassname@2020}{%
  \textup{2020} Mathematics Subject Classification}
\makeatother

\title[$\Gamma$-convergence of self-aggregation energies]{$\Gamma$-convergence of discrete energies modeling self-aggregation of stochastic particles}

\author[L. Lussardi]{Luca Lussardi}
\address[L. Lussardi]{Dipartimento di Scienze Matematiche ``G.~L.~Lagrange'',
Politecnico di Torino, Corso Duca degli Abruzzi 24,
10129 Torino, Italy.}
\email{luca.lussardi@polito.it}

\author[A. Melchor Hernandez]{Anderson Melchor Hernandez}
\address[A. Melchor Hernandez]{Dipartimento di Scienze Matematiche ``G.~L.~Lagrange'',
Politecnico di Torino, Corso Duca degli Abruzzi 24,
10129 Torino, Italy.}
\email{anderson.melchor@polito.it}

\author[M. Morandotti]{Marco Morandotti}
\address[Marco Morandotti]{Dipartimento di Scienze Matematiche ``G.~L.~Lagrange'',
Politecnico di Torino, Corso Duca degli Abruzzi 24,
10129 Torino, Italy.}
\email{marco.morandotti@polito.it}

\date{\today}

\begin{document}
% Stochastic and other probabilistic methods applied to problems in solid mechanics

\begin{abstract}
In this work, we demonstrate that a functional modeling the self-aggregation of stochastically distributed lipid molecules can be obtained as the $\Gamma$-limit of a family of discrete energies driven by a sequence of independent and identically distributed random variables. These random variables are intended to describe the asymptotic behavior of lipid molecules that satisfy an incompressibility condition. The discrete energy keeps into account the interactions between particles.  We resort to transportation maps to compare functionals defined on discrete and continuous domains, and we prove that, under suitable conditions on the scaling of these maps as the number of random variables increases, the limit functional features  an interfacial term with a Wasserstein-type penalization.
\end{abstract}

\maketitle

\tableofcontents

\section{Introduction}

Lipid bilayers are the fundamental component of cell membranes. 
They are composed of two layers of lipid molecules, with the hydrophilic heads facing outward and the hydrophobic tails facing inward. 
Lipid bilayers act as barriers, preventing the uncontrolled exchange of materials between the inside and outside of the cell \cite{blom2004}; 
they are dynamic and can change their shape and composition in response to various stimuli. 
For example, cells can regulate the fluidity of the membrane by altering the types of lipids, allowing them to respond to changes in temperature or other environmental factors \cite{jain2014}. 
Understanding the structure and functioning of lipid bilayers is essential for modeling basic biological processes such as cell signaling and membrane transport \cite{seifert1997}.

In the last decades, a vast body of literature has been devoted to describing different mathematical settings to study the properties of biomembranes in a rigorous manner. 
In this regard, Canham \cite{canham1970} and Helfrich \cite{helfrich1973} considered continuous models, where a lipid bilayer membrane assumes curved shapes to accommodate a spontaneous curvature which is influenced by factors such as lipid composition, temperature, and external stresses \cite{seifert1997}. 
In a similar setup, motivated by the works of Bates \cite{bates1990} and Fredrickson \cite{fredrickson1996} on chains of copolymers, the authors of \cite{blom2004} proposed a continuous model to tackle the problem of self-aggregation of lipid molecules, which was not addressed in Helfrich's seminal paper \cite{helfrich1973}. 
We also refer the reader to \cite{bailo2020,burger2018,carrillo2019,fetecau2011}, where pattern formation of cell aggregates is analyzed by means of kinetic models or aggregation-diffusion equations.

Lipid membranes can be studied at different length scales; in \cite{pelet2009}, the authors proposed three different models for micro-, meso, and macroscopic scales through suitable energy functionals for idealized and rescaled head-tail densities.
Their main result is the rigorous meso-to-macro limit in the two-dimensional setting, which recovers the Canham--Helrfich model.
In the same paper a formal derivation of the mesoscopic model was also proposed.
From that paper, a series of contributions stemmed, where the nature of curved membranes was analyzed \cite{lussardi2014,lussardi2016,peletier2010}. 
However, even though these models well describe self-aggregation of lipid molecules, they do not take into account the discrete nature of the problem. 
In fact, there are few works where the behavior of lipid molecules is modeled by functionals defined on discrete structures, see, e.g., \cite{gladbach2021,van2019}. 
One main reason for the lack of results in this direction is that capturing the contribution of the principal curvatures in a discrete setting is a subtle task. 

In this paper, we perform the rigorous micro-to-meso derivation that was proposed in \cite{pelet2009}.
To do so, we upscale (i.e., we take the limit as $n\to\infty$) discrete energies depending on a large number of molecules modeled by independent and identically distributed (according to a probability density $\rho$) random variables $X_1,\ldots,X_n$ ($n\in\N$) contained in some bounded region $D\subset\R^d$. 
Given $\delta_n>0$, we consider a rescaled kernel $\kappa_{\delta_n}(X_{i}-X_{j})$ that takes into account the interaction between lipid molecules at scale~$\delta_n$, for $i,j=1,\ldots, n$.
As the number~$n$ of molecules increases, the interaction length-scale $\delta_n$ must vanish due to the boundedness of the container~$D$, so that a crucial role in the identification of the mesoscopic limit will be played by the relationship between~$n$ and~$\delta_n$.
We encode this interaction in a functional $GF_{n,\delta}$ defined, for a general interaction length-scale $\delta$ (see the precise definition in \eqref{tilde} below), in such a way that the amphiphilic nature of the molecules is taken into account: more precisely, the vicinity of hydrophilic heads and water or the hydrophobic tails is favored, whereas the vicinity of the hydrophobic tails and water is penalized.

A major difficulty in taking the limit as $n\to\infty$ and to recover a continuous functional is due to the stochasticity of the system; a key ingredient in performing this limit is the used of transportation maps to set the problem in a continuous framework.
Our main result is Theorem~\ref{ourgamma} below where we show that, under a suitable rescaling of $\delta_{n}$ (see assumption (H3) below), the discrete $GF_{n,\delta_{n}}$ $\Gamma$-converges to a weighted total variation functional (originally introduced in \cite{baldi2001}). 
This means that, when modeling lipid bilayers using a sequence $X\coloneqq\{X_{i}\}_{i\in \N}$ of random variables that interact at the scales dictated by hypothesis (H3), the resulting mesoscopic energy is a surface functional which depends on the geometric properties of the interface between the heads and the tails of the amphiphilic molecules.

As a byproduct of our Theorem~\ref{ourgamma}, in the case of uniform distributed random variables, that is, when $\rho=|D|^{-1}$, we recover the usual perimeter functional formally derived in \cite{pelet2009}. 
More precisely, we consider the functional $\G_{n,\delta_{n}}$ defined in \eqref{ourfunct} below, given by the sum of $GF_{n,\delta_{n}}$ and a Wasserstein-like term and we prove in Theorem~\ref{mainresult} that $\G_{n,\delta_{n}}$ $\Gamma$-converges to the same mesoscopic functional of \cite{pelet2009}.
This is a significant result, which sets the ground for a derivation of the Canham--Helfrich functional starting from the more natural discrete stochastic setting.

%a sequence $X\coloneqq\{X_{i}\}_{i\in \N}$ of 

%Here, $\kappa$ is assumed to be radially symmetric satisfying a suitable decaying property (see hypothesis (H1) below). Moreover, we suppose that $\rho$ is supported in a set $D$ of $\R^{d}$ (see  hypothesis (H2) below). 

%Since $G_{n,\delta_{n}}$ is defined at the discrete level, we resort to transportation maps to compare $G_{n,\delta_{n}}$ with its continuum version. We show in 
The paper is organized as follows: in Section~\ref{sec:preli}, we recall the derivation of the model proposed in \cite{pelet2009} and some preliminary concepts needed later on. In  Section~\ref{hompin}, we introduce our main assumptions and state our main results, whose proofs are presented in Section~\ref{dimost}.

 %\subsection*{Summary of notation}
 
 %\begin{tabular}{lll}
 %$(\Omega,\T,\P)$& probability space&\\
 %$G_{n,\delta}(\cdot)$& discrete energy functional & \eqref{tilde}\\
 %$\G_{n,\delta}(\cdot,\cdot)$& discrete energy functional &\eqref{ourfunct}\\
 %$\F_{\delta}(\cdot,\cdot)$ & energy functional & \eqref{weighF}\\
 %$\X(D)$ & domain of $\G_{n,\delta}$ & \eqref{set0}\\
 %$W_{p}(\cdot,\cdot)$& Wasserstein distance of order $p$& \\
 %$T_{\#}\mu$& push-forward measure&\\
 %$\nu_{n}$& the empirical measure & \eqref{misnu}\\
 %$\Gamma(\mu,\nu)$& set of all couplings between $\mu$ and $\nu$\\
 %$TL^{p}(D)$& set of measure-function pairs& \eqref{TLspace}\\
%$d_{TL^{p}}$& distance on $TL^{p}(D)$ & \eqref{minimi}\\
%$\L^{d}$ & $d$-dimensional Lebesgue measure&\\
%$\H^{d-1}$ & (d-1)-dimensional Hausdorff measure\\
%$\partial^{\ast}B$ & reduced boundary of the set $B$ &\\
%$J_{u}$ & the jump set of $u$\\
%$TV(u;\psi)$ & weighted total variation of $u$ with respect to $\psi$& \eqref{weightototal}\\
%$\vert Du\vert_{\psi}$ & weighted total variation of $u$ with respect to $\psi$& \eqref{divthm}\\
%$\BV(D;\psi)$& set of all functions in $L^{1}(D;\psi)$ with finite weighted total variation with respect to $\psi$&\\
 %\end{tabular}
 
\section{Preliminaries}\label{sec:preli}
In  this part, we start by recalling the derivation of the mesoscopic model to treat bilayer molecules.
\subsection{Lipid models}
A lipid molecule consists of a head and two tails. The head is usually charged and hydrophilic, while the tails are hydrophobic. This difference in polarity causes lipids to aggregate, and inside a cell, lipids are typically found in a bilayer structure (see Figure \ref{imag1} below).

\begin{figure}[!htp]
\begin{center}
\includegraphics[width=0.4\textwidth]{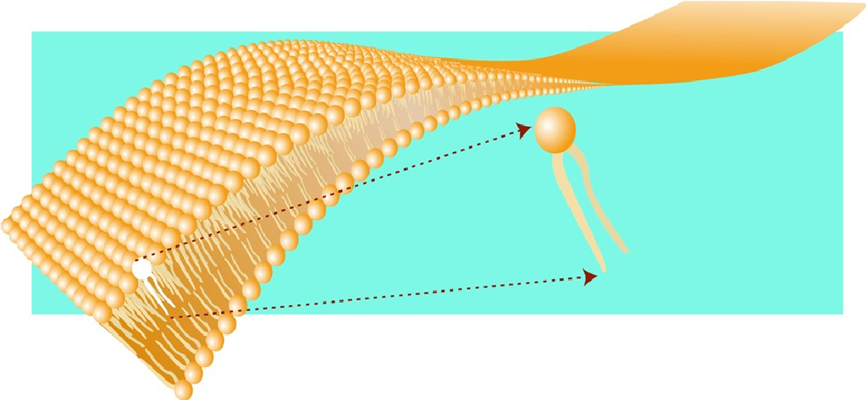}
\end{center}
\caption{A typical bilayer structure (see \cite[Figure~1]{blom2004}).}
\label{imag1}
\end{figure}
In such a structure, the energetically unfavorable tail-water interactions are avoided by collecting the tails together in a water-free region created by the heads. Suppose that we have a system of particles composed of $X^{t}_{i}$, $X^{h}_{i}$, and $X^{w}_{i}$, which are the lipid tail, lipid head, and water beads, respectively, for $i=1,\ldots, n$, where $n\in \N$. Assuming that the beads are confined to a space $D\subset \R^{d}$, the full microscopic state space for the system is then~$D^{3n}$. Elements $X\in D^{3n}$ are called microstates \cite{lussardi2016}. We describe the system in terms of probability densities on~$D^{3n}$; that is, the state $\psi$ of the system is described through a probability density on~$D^{3n}$
\begin{align*}
    \psi\in \cE,\quad \text{where}\quad \cE\coloneqq\bigg\{ \psi \colon D^{3n}\rightarrow \R_{+}: \int_{D^{3n}}\psi(x)\,\de x=1\bigg\}.
\end{align*}
Moreover, suppose that $X^{t}$, $X^{h}$, and $X^{w}$ are composed of independent and identically distributed random variables, which are distributed according to probability laws~$\mu$,~$\nu$, and~$\eta$, respectively, and that these distributions depend on~$\psi$. Suppose that~$\mu$,~$\nu$, and~$\eta$ are absolutely continuous with respect to the $d$-dimensional Lebesgue measure $\L^{d}$, and thus have densities~$u$,~$v$, and~$w$, respectively. A typical assumption when modeling biomembranes is that $u +v +w =1$.
This condition is usually referred to as the incompressibility condition for particles \cite[Appendix A]{pelet2009}. It is a point-wise relation that aims to balance the portion of particles of different types stored in~$D$. 
To specify the behavior of the system of particles, we consider two functionals: the \emph{ideal} free functional $F^{\rm id}\colon \cE\rightarrow \R$, $\psi\mapsto F^{\rm id}(\psi)$, which models the effects of entropy and the interaction between beads of the same type, and the \emph{non-ideal} free energy functional $F^{\rm nid}\colon \cE\rightarrow \R$, $\psi\mapsto F^{\rm nid}(\psi)$, which represents the interaction of beads of different types. The total energy of the system is then $F^{\rm nid}(\psi)+F^{\rm id}(\psi)$. Typical terms in the non-ideal free energy are convolution integrals in which proximity of hydrophilic beads and tail beads is penalized: up to a physical constant which we disregards, $F^{\rm nid}$ can be expressed as
\begin{align*}
F^{\rm nid}(\psi)\coloneqq\iint_{D\times D}(v(x) + w(x))u(y)\kappa(x-y)\,\de x\de y,
\end{align*}
where~$\kappa$ is the kernel of interaction. 
Since we are interested in interactions at a small scale $\delta>0$, the kernel $\kappa$ is usually rescaled to $\kappa_\delta \coloneqq \delta^{-d}\kappa(\cdot/\delta)$. 
For the ideal free energy $F^{\rm id}$, one can assume the form 
$$F^{\rm id}(\psi)= \int_{D^{3n}} \psi(x) H^{\rm id}(x)\,\de x,$$ 
where $H^{\rm id}$ is a suitable interaction potential. 
The choice of the potential~$H^{\rm id}$ models certain phenomena, such as attraction, repulsion, and combinations thereof. 
The function~$H^{\rm id}$ is usually called the \emph{interaction potential}, and it is in general assumed to be radially symmetric \cite{bailo2020,burger2018}, i.e., $H^{\rm id}(x)=w(\vert x\vert)$, for some $w\colon [0,+\infty)\rightarrow \R$. 
%\rosso{Typical choices include power laws in the theory of granular media, which feature short-range repulsion and long-range attraction in population dynamics \cite{carrillo2019,fetecau2011}.}  

In 2009, Peletier and R\"oger \cite{pelet2009} proposed a mesoscale model for biomembranes in the form of an energy functional for idealized and rescaled head-tail densities. 
Such a model is derived from the functional $F^{\rm nid}+F^{\rm id}$, in which heads and tails are treated as separate particles. 
The term $F^{{\rm nid}}$ penalizes the proximity of tails to either heads or water particles, and the term $F^{{\rm id}}$ implements the head-tail connection. %as an energetic penalization. 
In particular, they proposed $F^{{\rm id}}(\psi)= W_{p}(u,v)$, the Wasserstein distance of order $p\in [1,\infty)$. From an analytic point of view, configurations of head and tail particles are modeled by suitable density functions belonging to the space of functions of bounded variation in $\R^{d}$ \cite{ambrosio2000}. 
A formal upscaling procedure leads to the following model: configurations of head and tail particles are described by two rescaled density functions
 \begin{align*}
     u\in {\rm BV}(\R^{d};\{0,\epsilon^{-1}\}), \qquad v\in L^{1}(\R^{d}; \{0,\epsilon^{-1}\})
 \end{align*}
 with $uv=0$ a.e.~in $\R^{d}$ and with prescribed total mass $M>0$, namely
 \begin{align}\label{masscost}
     \int_{\R^d} u(x)\, \de x=\int_{\R^d} v(x)\,\de x=M.
 \end{align}
 Here~$\epsilon>0$ is a small parameter measuring the size of the support of~$u$. % which must be of order $O(\epsilon)$. 
 We denote by $K_{\epsilon}\subset L^{1}(\R^{d}\times \R^{d})$ the space of such configurations, that is
$$
K_{\epsilon}\coloneqq\big\{(u,v)\in  \BV(\R^{d},\{0,\epsilon^{-1}\})\times L^{1}(\R^{d}, \{0,\epsilon^{-1}\}): \text{\eqref{masscost} holds and $uv=0$ a.e.~in $\R^d$}\big\}.
$$
In \cite{pelet2009} the asymptotic behavior of  the energy functional
\begin{align}\label{mesoscopic}
F_{\epsilon}(u,v)\coloneqq
     \begin{cases}
     \displaystyle \epsilon\norma{Du} +\frac{1}{\epsilon}W_{1}(u,v) & \text{if $(u,v)\in K_{\eps}$,}\\
     +\infty & \text{otherwise in $L^{1}(\R^{d})\times L^1(\R^{d})$}
     \end{cases}
 \end{align}
is studied.
%Here, $u$ corresponds to the tail density, while~$v$ is the density of heads \rosso{non si era già detto?}. 
The term~$\norma{Du}$ represents the total variation of~$u$
%(up to the constant $\epsilon$), 
and it measures the boundary of the support of tails. 
By a rigorous analysis via $\Gamma$-convergence, it is shown in~\cite{pelet2009} that the energy \eqref{mesoscopic} has a preference for thin structures without ends and a resistance to bending of the structure. 
To be precise, in the limit as $\epsilon\to0$, the densities concentrate along families of regular curves, and a Euler-type functional appears. 
The result is in the two-dimensional case, whereas in the three-dimensional case one might expect convergence to a Canham--Helfrich-type functional; nonetheless, in this situation, only partial results are known \cite{lussardi2014, lussardi2016}.

\subsection{$\Gamma$-convergence}
In this section, we briefly review the definition of $\Gamma$-convergence of a sequence of functionals defined on a metric space, both in the deterministic and in the stochastic case.
Let $(Y,d_{Y})$ be a metric space and let $F_{n}\colon Y\rightarrow [0,\infty]$ be a sequence of functionals.

\begin{definition}\label{congam}
The sequence $\{F_{n}\}_{n\in\N}$ $\Gamma$-converges with respect to metric $d_{Y}$ to the functional $F\colon Y\rightarrow [0,\infty]$ if the following inequalities hold:
\begin{enumerate}
\item[i)] For every $y\in Y$ and every sequence $\{y_{n}\}_{n\in\N}$ converging to $y$
\begin{align*}
F(y)\leq \liminf_{n\rightarrow \infty}F_{n}(y_{n});
\end{align*}
\item[ii)] For every $y\in Y$ there exists a sequence $\{\bar{y}_{n}\}_{n\in\N}$ converging to $y$ such that
\begin{align*}
\limsup_{n\rightarrow \infty}F_{n}(\bar{y}_{n})\leq F(y).
\end{align*}
\end{enumerate}
In this case, we then say that $F$ is the $\Gamma$-limit of the sequence of functionals $\{F_{n}\}_{n\in\N}$ with respect to the metric $d_{Y}$. Since the sequence $\{\bar{y}_{n}\}_{n\in\N}$ in $ii)$  satisfies condition i) we have $\lim_{n\rightarrow \infty}F(\bar{y}_{n})=F(y)$, and the sequence $\{\bar{y}_{n}\}_{n\in\N}$ is called recovery sequence.
\end{definition}
In the same way, we define the  $\Gamma$-convergence for random functionals.
\begin{definition}
Let $(\Omega,\T,\P)$ be a probability space. 
For a sequence of random functionals $F_{n}\colon Y\times \Omega\rightarrow [0,+\infty]$ and $F\colon Y\rightarrow [0,+\infty]$ a deterministic functional, we say that the sequence of functionals $\{F_{n}\}_{n\in\N}$ $\Gamma$-converges with respect to the metric $d_{Y}$ to $F$, if for $\P$-almost every $\omega\in\Omega$, the sequence $\{F_{n}(\cdot,\omega)\}_{n\in\N}$ $\Gamma$-converges to~$F$ according to Definition \ref{congam}.  
\end{definition}
We now recall  the notion of convergence in law and convergence in probability.
Let $Y$ and $Z$ be metric spaces. 
Given a probability measure $\mu\in \cP(Y)$ and a measurable mapping $T\colon Y\rightarrow Z$, we denote by $\nu \coloneqq T_{\#}\mu$ the push-forward of~$\mu$ by~$T$, namely the measure $\nu\in \cP(Z)$ such that $\nu(A)=(T_{\#}\mu)(A)=\mu(T^{-1}(A))$, for any Borel set $A\subset Z$.
\begin{definition}
Let $\{F_{n}\}_{n\in\N}$ be a sequence of random functionals with~$F_{n}$ defined on the probability space $(\Omega_{n},\T_{n},\P_{n})$. 
Let~$F_{\infty}$ be a random  functional defined on the probability space $(\Omega_{\infty},\T_{\infty},\P_{\infty})$.
We say that $\{F_{n}\}_{n\in\N}$ \emph{converges in law} to~$F_{\infty}$ if the sequence of measures $\mu_{n}\coloneqq (F_{n})_{\#}\P_{n}$ converges weakly-* as $n\rightarrow \infty$ to $\mu_{\infty}\coloneqq (F_{\infty})_{\#}\P_{\infty}$.
\end{definition}
\begin{definition}
Let $\{F_{n}\}_{n\in\N}$ be a sequence of random functionals defined on a common probability space $(\Omega,\T,\P)$. 
We say that~$F_{n}$ \emph{converges in probability} to a random functional~$F_{\infty}$ if 
\begin{align*}
\lim_{n\rightarrow \infty}\P\left(\{\omega\in \Omega: d(F_{n},F_{\infty})>\eta \}\right)=0, \quad\text{for all $\eta>0$}
\end{align*}
where $d$ is a suitable metric in the space of functionals (see \cite[Proposition~1.12]{modica1986}). 
\end{definition}
A well-known result about the convergence of random functionals is the following one.
\begin{proposition}[{\cite[Proposition 2.9]{modica1986}}]\label{DMM_p29}
Suppose that $\{F_{n}\}_{n}$ is a sequence of random functionals which converges in law to $F_{\infty}\colon Y\times \Omega\rightarrow [0,+\infty]$. 
If~$F_{\infty}$ is constant, that is, if there exists $F\colon Y\to [0,+\infty]$ such that $F_{\infty}(\omega)=F$ for  $\P$-almost every $\omega\in \Omega$, then the convergence  of~$\{F_{n}\}_{n}$ in law and the convergence in probability are equivalent.
\end{proposition}

\subsection{The $TL^{p}$ topology and some of its properties}\label{mappe}
In this section, we recall some useful facts about the topology of measure-functions pairs spaces. This topology was introduced in \cite{garcia2016} to treat the convergence of functionals defined on graphs. Given an open bounded subset $D\subset \R^{d}$, we define for all $p\in[1,+\infty)$
\begin{align*}%\label{TLspace}
TL^{p}(D)\coloneqq \{(\mu,f):\mu\in \cP(D), f\in L^{p}_\mu(D)\}.   
\end{align*}
This space can be endowed with a distance (see \eqref{minimi} below) and thus it is a metric space \cite[Proposition 3.3]{garcia2016}.\footnote{The definition for the case $p=\infty$ can be presented in the usual way; we neglect it here for the sake of conciseness.} 
Note that for a generic bounded Borel map $f\colon D\rightarrow \R$, the following change of  variables
\begin{align*}%\label{transp}
\int_{D}f(T(x))\,\de\mu(x)=\int_{D}f(y)\,\de\nu(y),  
\end{align*}
holds true if and only if $\nu=T_{\#}\mu$, in which case we say that the Borel map $T\colon D\rightarrow D$ is a \emph{transportation map} between the measures $\mu\in \cP(D)$ and $\nu\in \cP(D)$. 
In this case, we associate with~$T$ a transportation plan $\pi_{T}\in \Gamma(\mu,\nu)$ defined by 
\begin{align}\label{plan1}
 \pi_{T}\coloneqq ({\rm Id},T)_{\#}\mu   
\end{align}
where $({\rm Id},T)\colon D\rightarrow D\times D$ is given by 
$({\rm Id},T)(x)=(x,T(x))$. 
Here, $\Gamma(\mu,\nu)$ is the set of all couplings between $\mu$ and $\nu$, that is, the set of all Borel probability measures on $D\times D$ for which the marginal on the first variable is~$\mu$ and the marginal on the second variable is~$\nu$.
%Note that for any $c\in L^{1}\left(D\times D;\pi_{T}\right)$
%\begin{equation}\label{traspt}
%\int_{D\times D}c(x,y)\,\de\pi_{T}(x,y)=\int_{D}c(x,T(x))\,\de\mu(x).
%\end{equation}

We notice that when~$\mu$ is absolutely continuous with respect to the Lebesgue measure, the weak-* convergence of~$\mu_{n}$ to~$\mu$ as $n\rightarrow \infty$ is equivalent to the existence of a sequence $\{T_{n}\}_{n\in \N}$ of transportation maps such that
\begin{align}\label{stagning}
\int_{D}\vert x-T_{n}(x)\vert^{p}\,\de\mu(x) \rightarrow 0 \qquad \text{as $n\rightarrow \infty$.}    
\end{align}
This motivates the following definition.
\begin{definition}\label{def_stagn}
We say that a sequence of transportation maps $\{T_{n}\}_{n\in\N}$ is \emph{stagnating} if %$(T_{n})_{\#}\mu=\mu_{n}$ and 
\eqref{stagning} holds.
\end{definition}
Given $p\in[1,\infty)$ and $(\mu,f),(\nu,g)\in TL^{p}(D)$ their distance is defined by
\begin{align}\label{minimi}
d_{TL^{p}}\big((\mu,f),(\nu,g)\big)\coloneqq 
%\begin{cases}
%\displaystyle 
\inf_{\pi\in \Gamma(\mu,\nu)}  \iint_{D\times D}  (\vert x-y \vert^{p} + \vert f(x)-g(x)\vert^{p}) \,\de\pi(x,y)
%\hfill \text{if $p\in [1,\infty)$,}\\
%\displaystyle \inf_{\pi\in \Gamma(\mu,\nu)}{\rm ess\,sup}_{(x,y)\in \spt(\pi)}\left(\vert x-y \vert+\vert f(x)-g(y)\vert \right)\\
%\hfill \text{if $p=\infty$.}
%\end{cases}
\end{align}
It is well known \cite{amb2005,santambrogio2015} that when~$\mu$ is absolutely with respect to the Lebesgue measure, problem \eqref{minimi} admits a solution. In particular, if $p>1$ there exists a unique transport map~$T$ inducing the transport plan~$\pi_T$ in \eqref{plan1}. 

The following proposition characterized the convergence in $TL^p$.
\begin{proposition}[{\cite[proposition 3.12]{garcia2016}}]\label{prop2.7}
Let $(\mu,f)\in TL^{p}(D)$ and let $\{(\mu_{n},f_{n})\}_{n\in\N}$ be a sequence in $TL^{p}$. The following statements are equivalent:
\begin{enumerate}
    \item[1.] $(\mu_{n},f_{n})$ converges to $(\mu,f)$ in $TL^p(D)$ as $n\to\infty$, in symbols $(\mu_{n},f_{n})\stackrel{TL^{p}}{\longrightarrow} (\mu,f)$ as $n\rightarrow \infty$. 
    \item[2.] $\mu_{n}\weak\mu$ weakly$^{\ast}$ and for every stagnating sequence of transportation plans $\{\pi_{n}\}_{n\in\N}\subseteq \Gamma(\mu,\mu_{n})$
    \begin{align}\label{stagn1}
        \iint_{D\times D} \vert f(x)-f(y) \vert^{p} \, \de \pi_{n}(x,y) \rightarrow 0, \quad \text{as $n\rightarrow \infty$.}
    \end{align}
    \item[3.] $\mu_{n}\weak \mu$ weakly$^{\ast}$ and there exists a stagnating sequence of transportation plans $\{\pi_{n}\}_{n\in\N}\subseteq \Gamma(\mu,\mu_{n})$ for which \eqref{stagn1} holds.
\end{enumerate}
Moreover, if the measure $\mu$ is absolutely continuous with respect to the Lebesgue measure, the following are equivalent to the previous statements:
\begin{enumerate}
    \item[4.]$\mu_{n}\weak \mu$ weakly$^{\ast}$ and there exists a stagnating sequence of transportation maps $\{T_{n}\}_{n\in\N}$ such that
    \begin{align}\label{stagn2}
     \int_{D}\vert f(x)-f_{n}(T_{n}(x)) \vert^{p}\,\de\mu(x)\rightarrow 0, \quad \text{as $n\rightarrow \infty$;}    
    \end{align}
    \item[5.] $\mu_{n}\weak \mu$ weakly$^{\ast}$ and for any stagnating sequence of transportation maps $\{T_{n}\}_{n\in\N}$ \eqref{stagn2} holds.
\end{enumerate}
\end{proposition}

%In the following, we will assume that $D\subset\R^d$ is a connected set with Lipschitz boundary. 
Let $\nu=\rho\L^{d}$ with the density $\rho$ bounded from below and from above by positive constants. 
We consider a sequence $\{X_{i}\}_{i\in\N}$ of independent %and identically distributed 
random variables which are identically distributed according to~$\nu$, and supported on a common probability space $(\Omega,\T,\P)$. 
In the next, we make use of upper bounds on the transportation distance between~$\nu$ and the empirical measure $\nu_{n}\coloneqq \frac{1}{n}\sum_{i=1}^{n}\delta_{X_{i}}$. 
This will play an important role in the proof of our $\Gamma$-convergence results. 
In particular, we will use estimates of
\begin{align*}
d_{\infty}(\nu,\nu_{n})\coloneqq\inf\big\{\norma{{\rm Id}-T}_{\infty} \text{ where $T\colon D\rightarrow D$, $T_{\#}\nu=\nu_{n}$}\big\}
\end{align*}
which measures what is the least maximal distance that a transportation map~$T$ between~$\nu$ and~$\nu_{n}$ has to move the mass. 
Here, for a map $S\colon D\to D$, we denote by $\lVert S\rVert_{\infty}\coloneqq \sup_{x\in D}|S(x)|$.
In particular, we are interested in the case where~$\rho$ is a uniform probability density in~$D$. 
Let us suppose that $D=(0,1)^{d}$, and consider $P=\{p_{i},\ldots, p_{n}\}$ the set of~$n$ points in~$D$ where each point in~$P$ is the center of a cube contained in~$D$ with volume~$\frac{1}{n}$. In \cite{leighton1989} it is shown that, if~$D$ is connected with Lipschitz boundary, for $d\geq 3$ and $\rho$ constant, there exist two positive constants $\lambda, \Lambda$ such that the event
\begin{align*}
\frac{\lambda (\log n)^{1/d}}{n^{1/d}}\leq \min_{\sigma}\max_{i\in\{1,\ldots,n\}}\vert p_{i}-X_{\sigma(i)}\vert \leq \frac{\Lambda (\log n)^{1/d}}{n^{1/d}}
\end{align*} 
has probability $1$ and where $\sigma$ ranges over all permutations of $\{1,\ldots,n\}$. This result was further extended in \cite{trillos2015}, and we report it here.
\begin{theorem}[{\cite[Theorem 2.5]{garcia2016}}]\label{trillos15}
Let~$D\subset\R^d$ be an bounded connected open set with Lipschitz boundary. 
Let~$\nu$ be a probability measure on~$D$ with density~$\rho$ which is bounded from below and from above by positive constants. 
Let $\{X_{i}\}_{n\in\N}$ be a sequence of independent and identically distributed random variables distributed on~$D$ according to the measure~$\nu$ and let $\nu_{n}\coloneqq \frac{1}{n}\sum_{i=1}^{n}\delta_{X_{i}}$. 
Then there exists a constant $C>0$ such that for $\P$-almost everywhere $\omega\in\Omega$ there exists a sequence of transportation maps $\{T_{n}\}_{n\in\N}$ from $\nu$ to $\nu_{n}$ (i.e., $(T_{n})_{\#}\nu=\nu_{n}$) and such that
\begin{align}\label{estimaL}
\begin{aligned}
&\text{if $d=2$ then} \quad \limsup_{n\in\N}\frac{n^{1/2}\norma{{\rm Id}-T_{n}}_{\infty}}{(\log n)^{3/4}}\leq C,\\
&\text{if $d\geq 3$ then} \quad \limsup_{n\in\N}\frac{n^{1/d}\norma{{\rm Id}-T_{n}}_{\infty}}{(\log n)^{1/d}}\leq C.
\end{aligned}
\end{align}
\end{theorem}
\begin{remark}\label{iofareiunremark}
As a consequence of \eqref{estimaL}, we have that $\norma{{\rm Id}-T_{n}}_{\infty}\to0$. 
In particular, the sequence $\{T_n\}_{n\in\N}$ is stagnating according to Definition~\ref{def_stagn}.
\end{remark}

\subsection{Weighted {\rm BV} functions}\label{sec_wBV}
We now recall the notion of total variation and weighted total variation which will be used to state our main result. 
%As customary, we will denote by~$\BV(D)$ the space of functions of bounded variation defined on~$D$ (see \cite{ambrosio2000}). 
Let $\psi\colon D\rightarrow (0,+\infty)$ be a continuous function, and consider the measure $\nu=\psi \L^{d}$.  
Following \cite{baldi2001}, given $u\in L^{1}_\nu(D)$ we define the weighted total variation of~$u$ with respect to~$\psi$ as
\begin{equation*}%\label{weightototal}
TV(u;\psi)\coloneqq\sup\bigg\{\int_{D}u \div  \phi \,\de x:  \phi\in C_{c}^{\infty}(D;\R^{d}),\vert\phi(x) \vert\leq \psi(x)  \text{ for all $x\in D$} \bigg\}.
\end{equation*}
We denote $\BV(D;\psi)$ the set of all functions $u\in L^{1}(D;\psi)$ for which $TV(u;\psi)<+\infty$. 
In particular, when $\psi\equiv 1$ we recover the usual space~$\BV(D)$. 
For a measurable set $B\subset D$, we then define the perimeter in~$D$ as the weighted total variation of the characteristic function of~$B$, that is, $\PR(B;\psi) \coloneqq TV(\chi_{B};\psi)$. 
\begin{proposition}[{\cite[Theorem 3.3]{baldi2001}}]
A function $u\in L^{1}(D;\psi)$ belongs to~$\BV(D;\psi)$ if and only if there exist a finite Radon measure $\vert Du\vert_{\psi}$ and a $\vert Du\vert_{\psi}$-measurable function $\sigma\colon D\rightarrow \R^{d}$ such that $\vert \sigma(x)\vert=1$ for $\vert Du\vert_{\psi}$-almost every~$x\in D$ and such that
\begin{align}\label{divthm}
\int_{D} u(x) \div\phi(x)\,\de x=-\int_{D}\frac{\phi(x)\cdot \sigma(x)}{\psi(x)}\,\de\vert Du\vert_{\psi}(x).
\end{align}
The measure~$\vert Du\vert_{\psi}$ and the function~$\sigma$ are uniquely determined by \eqref{divthm} and the weighted total variation $TV(u;\psi)$ is equal to $\vert Du\vert_{\psi}(D)$.
\end{proposition}
Note that, using \eqref{divthm}, one can check that $\vert Du\vert_{\psi}= \psi \vert Du\vert$, so that
\begin{align}\label{formtv}
TV(u;\psi)= \int_{D}\psi(x) \,\de \vert Du\vert(x).
\end{align}
%\rosso{Denoting by $\partial^{\ast}B$ the reduced boundary of~$B$, in view of~\eqref{divthm} we have
%\begin{align*}
%\PR(B;\psi)=\int_{D\cap \partial^{\ast}B} \psi(x)\,\de\H^{d-1}(x).
%\end{align*}
%}
Since the functional $TV(\cdot;\psi)$ is defined as the supremum of linear continuous functionals in $L^{1}(D;\psi)$, 
%then $TV(\cdot;\psi)$ 
it is lower semicontinuous with respect to the $L^{1}(D;\psi)$ metric.  
The following density theorem for weighted {\rm BV} functions will be used in the proof of Theorem \ref{ourgamma} below.
\begin{theorem}[{\cite[Theorem 2.4]{trillos2015}}]\label{densbv}
Let $D$ be an open subset of $\R^{d}$ with Lipschitz boundary and let $\rho\colon D\rightarrow \R$ be a continuous function which is bounded from below and above by positive constants. 
Then for every $u\in {\rm BV}(D;\rho)$ there exists a sequence $\{u_{n}\}_{n\in\N}\subseteq C_{c}^{\infty}(\R^{d})$ such that $u_{n}\rightarrow u$ in $L^{1}(D)$ and $\int_{D}\vert\nabla u_{n} \vert \rho \,\de x \rightarrow TV(u;\rho)$ as $n\rightarrow \infty$.
\end{theorem}
For a general statement for density of regular functions in the space ${\rm BV}(D;\rho)$, we refer to \cite[Theorem 3.4]{baldi2001}.

\section{Setting of the problem and main results}\label{hompin}
In this section, we give a precise description of the functionals we are interested in and a precise statement of our main results.  
Let $X\coloneqq \{X_{i}\}_{i\in\N}$ be a sequence of independent and identically distributed random variables with distribution~$\nu$ which is absolutely continuous with respect to the Lebesgue measure in~$\R^{d}$, that is there exists a function $\rho\in L^{1}(D)$ such that $\nu=\rho\L^{d}$. 
Let $\kappa\colon \R^{d}\rightarrow [0,+\infty)$ be a %an even 
function % such that
% \begin{align*}
%     \int_{\R^{d}}\kappa(x)(1+\vert x\vert)\,\de x <+\infty
% \end{align*}
and, for any $\delta>0$, let $\kappa_{\delta}\colon \R^{d}\rightarrow [0,+\infty)$ be defined as
\begin{align*}
    \kappa_{\delta}(x)\coloneqq \frac{1}{\delta^{d}}\kappa\left(\frac{x}{\delta}\right).
\end{align*}
For %any $\delta>0$ and 
$n\in \N$, we consider the non-ideal free energy functional  $GF_{n,\delta}\colon L^{1}_\nu(D)\rightarrow [0,+\infty]$ defined by
\begin{align}\label{tilde}
GF_{n,\delta}(u)\coloneqq
\begin{cases}
\displaystyle \frac{2}{\delta n^{2}}\sum_{i,j=1}^n\kappa_{\delta}(X_{i}-X_{j})(1-u(X_{i}))u(X_{j}) & \text{if $0\leq u \leq 1$},\\[3mm]
+\infty & \text{otherwise.}
\end{cases}
\end{align}

Our contribution in this work is to study the asymptotic behavior of $GF_{n,\delta_{n}}$, for a suitable scaling $\delta_{n}$ depending on the number of particles $n$.

The factor $2/\delta $ in~\eqref{tilde} comes from a modeling assumption, whereas the factor $1/n^2$ is a rescaling that keeps into account the number of particles. 
The scaling $\delta_n\coloneqq \delta(n)$ that we are going to consider is the one that will allow us to obtain a law of large numbers as $n\rightarrow \infty$. %, and thus some hypotheses on~$\delta$ are needed. 
%Note that  $GF_{n,\delta}$ is a random functional which penalizes interactions between random variables at small scales. 
Our main result will be to show that the $\Gamma$-limit of $GF_{n,\delta_{n}}$ is $\vert Du\vert_{\rho^{2}}$ (see Section~\ref{sec_wBV}), and we prove this in Theorem \ref{ourgamma} below.

For $d>1$, we make the following assumptions which will be valid for the remainder of the paper.
\begin{enumerate}[({H}1)]
\item\label{H1} The set $D\subset\R^d$ is bounded, connected, and open and has Lipschitz boundary.
\item\label{H2} The kernel $\kappa\colon \R^{d}\rightarrow [0,+\infty)$ is isotropic and thus it can be written as $\kappa(x)=\eta(\vert x\vert)$, for a certain radial profile $\eta\colon [0,+\infty)\rightarrow [0,+\infty)$. We assume that~$\eta$ is non-increasing, continuous at $0$, and that $\eta(0)>0$, and also that 
\begin{align}\label{alberti}
\int_{0}^{+\infty}\eta(r)r^{d+1}\,\de r<+\infty.
\end{align}
\item\label{H3} $X=\{X_i\}_{i\in\N}$ is a sequence of independent and identically distributed random variables, distributed according to a probability distribution $\nu=\rho\L^{d}$ with Lipschitz continuous density $\rho\colon D\rightarrow \R$ which is bounded from below and above by positive constants~$a$ and~$b$, respectively.
Associated with~$X$, we consider the sequence of empirical measures
\begin{align}\label{misnu}
\nu_{n}\coloneqq \frac{1}{n}\sum_{i=1}^{n}\delta_{X_{i}}\,,
\end{align}
and by~$T_n$ the transportation maps such that $(T_n)_\#\nu=\nu_n$.
\item\label{H4} We let $\{\delta_{n}\}_{n\in\N}$ be an infinitesimal sequence of positive numbers satisfying
\begin{align}\label{delta's}
\begin{aligned}
  &\lim_{n\rightarrow \infty}\frac{(\log n)^{3/4}}{n^{1/2}\delta_n}=0 \qquad \text{if $d=2$,}\\
 &\lim_{n\rightarrow \infty}\frac{(\log n)^{1/d}}{n^{1/d}\delta_n}=0 \qquad \text{if $d\geq 3$.}
 \end{aligned}
\end{align}
\end{enumerate}
\begin{remark}\label{remark31}
Let us now give some comments about our assumptions. 
\begin{enumerate}
\item Hypothesis (H2) is physically motivated and it is needed to treat, among other problems, Ising spin systems on graphs \cite{alberti1998}. 
By elementary considerations, it is immediate to notice that there exists a step function $\eta_0\colon[0,+\infty)\to[0,+\infty)$ of the form
\begin{subequations}\label{eta_0}
\begin{equation}\label{eta_0a}
\eta_0(r)=\begin{cases}
A & \text{if $r< r_0$,}\\
0 & \text{if $r\geq r_0$,}
\end{cases}
\end{equation}
where $0<A<\eta(0)$ and $0<r_0<+\infty$, such that 
\begin{equation}\label{eta_0b}
\eta_0(r)\leq \eta(r)\qquad\text{for all $r\in[0,+\infty)$.}
\end{equation}
\end{subequations}
Clearly, $\eta_0$ satisfies \eqref{alberti}.
\item Hypothesis (H3) describes how each lipid molecule moves randomly in the environment independently of all other molecules. 
\item Hypothesis (H4) is crucial to obtain the compactness result in Lemma~\ref{compatt} below; in particular, the conditions \eqref{delta's} are quite technical and are related to upper bounds on the transportation distance between the empirical measure~\eqref{misnu} and the distribution~$\nu$. It is related to the connectedness property of the limiting graph obtained starting from a random graph with edges $X_{1},\ldots, X_{n}$ and edge weights $\kappa_{\delta_{n}}(X_{i}-X_{j})$ as $n\to\infty$. That is, the conditions \eqref{delta's} guarantee that the resulting graph is connected with probability $1$ as $n\rightarrow \infty$ (see \cite{gupta1999}).
\end{enumerate}
\end{remark}
\begin{lemma}(Compactness)\label{compatt}
Assume that (H\ref{H1}--\ref{H4}) hold true. Consider a sequence of functions $\{u_n\}_{n\in\N}$ such that $u_{n}\in L^{1}_{\nu_n}(D)$ for every $n\in\N$,  where~$\nu_{n}$ is given by~\eqref{misnu} and suppose that
\begin{equation}\label{HP_CPT}
\sup_{n\in \N}\norma{u_{n}}_{L^{1}_{\nu_n}(D)}<+\infty \qquad\text{and}\qquad
\sup_{n\in\N} GF_{n,\delta_{n}}(u_{n})<+\infty.
\end{equation}
Then $\{(\nu_n,u_{n})\}_{n\in\N}$ is relatively compact in $TL^{1}(D)$.
%, \rosso{i.e., there exists a function $u\in L^1(D;\nu)$ such that (up to a subsequence) $(\nu_n,u_n)\stackrel{TL^{p}}{\longrightarrow} (\nu,u)$}.
\end{lemma}
We now state the main result of this paper. %that the $\Gamma$-limit of  $G_{n,\delta_{n}}$  is  $\vert Du\vert_{\rho^{2}}$. 
To this aim, we define 
\begin{align}\label{costant}
  \alpha_{d}\coloneqq \bigg(\fint_{\S^{d-1}}\vert \nm{z,e}\vert \,\de \H^{d-1}(z)\bigg){\cdot}\bigg(\int_{\R^{d}}\vert x\vert \kappa(x)\,\de x\bigg),
\end{align}    
where $e\in\S^{d-1}$ any unit vector (it is not difficult to see that the averaged integral is indeed independent of~$e$).
%
%\begin{align*}
%K_{1,d}\coloneqq \fint_{\S^{d-1}}\vert \nm{z,e}\vert \,\de \H^{d-1}(z) \qquad\text{and}\qquad \theta_{d}\coloneqq \int_{\R^{d}}\vert x\vert \kappa(x)\,\de x,
%\end{align*}
\begin{theorem}\label{ourgamma}
Assume that (H\ref{H1}--\ref{H4}) hold true. % and let $\{\delta_n\}_{n\in\N}$ be as in (H4).
%$\nu_{n}$ be defined as in \eqref{misnu} and 
Then the functionals 
$GF_{n,\delta_n}\colon L^{1}_{\nu_n}(D)\rightarrow [0,+\infty]$ defined by \eqref{tilde} %. Then the functional $GF_{n,\delta_{n}}$ 
$\Gamma(TL^1(D))$-converge to the functional $\alpha_d TV(\cdot; \rho^{2})\colon L^1_\nu(D)\to[0,+\infty]$, where $\alpha_d$ is defined in \eqref{costant} and where, using \eqref{formtv},
\begin{align*}%\label{TVrho2}
TV(u;\rho^{2})=\begin{cases}
\displaystyle %\int_{D}\rho^2(x)\, \de\vert Du\vert(x) =\red{
\int_{S_u} \rho^2(x)\,\de\cH^{d-1}(x)& 
	\begin{array}{l}
	\text{if $u\in {\rm BV}(D;\rho)$}\\
	\text{and $u(x)\in\{0,1\}$ for $\nu$-a.e.~$x\in D$,}
	\end{array}\\[4mm]
+\infty & \;\;\text{otherwise in $L^1_\nu(D)$.}
\end{cases}
\end{align*}
%\rosso{and 
%\begin{align}\label{rightcost}
%\sigma_{d}\coloneqq \int_{\R^{d}}\kappa(x)\vert x\cdot e\vert \,\de x,
%\end{align}
%for $e$ an arbitrary but fixed unit vector of $\R^{d}$.}
\end{theorem}
%Here the $\Gamma$-convergence of $GF_{n,\delta_{n}}$ towards a deterministic limit functional is given for $\P$-almost every realization $\omega$. Since the limit functional is deterministic, then the convergence in probability and convergence in law are equivalent \cite[Proposition 2.9]{modica1986}. 

The next two results that we present hold in the case of constant density $\rho=|D|^{-1}$.
For $D\subset\R^{d}$ such that $|D|\geq2$,  consider the class
\begin{align*}%\label{set0}
\X(D)\coloneqq \big\{ (A,O)\in \B(D)\times \B(D): \vert A\vert =\vert O\vert=1, \vert A\cap O\vert=0\big\}.
\end{align*}
%of all pair of measurable sets which are disjoint and with unit volume. 
We study the asymptotic behavior as $\delta_{n} \rightarrow 0$ of the sequence of functionals $\G_{n,\delta_{n}}\colon \B(D)\times \B(D)\rightarrow [0,+\infty]$ defined as
\begin{align}\label{ourfunct}
\G_{n,\delta_{n}}(A,O)\coloneqq
\begin{cases}
GF_{n,\delta_{n}}(\chi_{A}) +W_{p}(A,O) & 
%\!\!\!\!\!\! \begin{array}{ll}
\text{if $(A,O)\in \X(D)$,} \\
%\text{\red{is a set of finite perimeter}, }
%\end{array} \\[4mm]
+\infty  &\text{otherwise.}
\end{cases}
\end{align}
Here, we use the notation $W_{p}(A,O)$ to denote the $p$-th Wasserstein distance between the indicator functions $\chi_{A}$ and $\chi_{O}$ of two sets~$A,O\in\B(D)$ such that $|A|=|O|=1$. 
\begin{lemma}\label{corb1}
Assume that (H\ref{H1}--\ref{H4}) hold true with $\rho=\vert D\vert^{-1}$.
Let $\{(A_{{n}},O_{{n}})\}_{{n}}\subset \X(D)$ be a sequence such that 
\begin{equation}\label{ultima}
\G_{n,\delta_n}(A_{n},O_{n})\leq  \inf_{(\widetilde A,\widetilde O)\in\X(D)}\G_{n,\delta_{n}}(\widetilde A,\widetilde O)+\frac1n.
\end{equation}
Then there exists $(A,O)\in\X(D)$, with~$A$ a set of finite perimeter, such that (up to subsequence) 
\begin{equation}\label{hypothesis}
\chi_{A_{n}}\to\chi_A\quad\text{in $L^1(D)$}\qquad\text{and}\qquad \chi_{O_{n}}\stackrel*\rightharpoonup \chi_O\quad\text{in $L^\infty(D)$}\qquad\text{as $n\to\infty$.}
\end{equation}
%with $A$ a set of finite perimeter, such that
%and there exists an infinitesimal sequence $\{\eta_n\}_n$ such that 
%\begin{equation}\label{ultima}
%\G_{n,\delta_n}(A_{n},O_{n})\leq  \inf_{(A,O)\in\X(D)}\G_{n,\delta_{n}}(A,O)+\eta_n.
%\end{equation}
%Then there exists a subsequence of $\{\delta_{n}\}_{n\in\N}$ (not relabeled) and $(A,O)\in \X(D)$ such that $A_{\delta_{n}}\rightarrow A$ in $L^{1}(D)$ as $n\rightarrow \infty$. Furthermore, $A$ is a set of finite perimeter.
\end{lemma}
Finally, we will prove the following $\Gamma$-convergence result.
\begin{theorem}[Gamma-convergence]\label{mainresult}
Assume that (H\ref{H1}--\ref{H4}) hold true with $\rho=\vert D\vert^{-1}$. Then
\begin{align*}%\label{MRGc}
    \Big(\Gamma-\lim_{n\rightarrow \infty}\G_{n,\delta_{n}}\Big)(A,O)=\frac{\alpha_d}{|D|^2}%\H^{d-1}(\partial^{\ast}A\cap D)
    {\rm Per}(A;D)+ W_{p}(A,O),
\end{align*}
where $\alpha_d$ is defined in~\eqref{costant}, and the metric of the $\Gamma$-convergence is $($s-$L^1(D))\times($w$^*$-$L^\infty(D))$.
\end{theorem}
%Note that the constant~$\alpha_d$ depend in a linear fashion the kernel~$\kappa$. 
%This property will be important in the proof of Theorems~\ref{ourgamma} and~\ref{mainresult}.

\begin{remark}
In view of Proposition~\ref{DMM_p29} we point out that both the $\Gamma$-convergence results in Theorems~\ref{ourgamma} and~\ref{mainresult} hold also in law.
\end{remark}

\section{Proofs}\label{dimost}
We start by proving our compactness result. 
%and $\Gamma$-convergence results for the functionals defined in \eqref{tilde}.
\begin{proof}[Proof of Lemma \ref{compatt}]
%Note that, since the topology of $\BV(D)$ and $\BV(D;\rho^{2})$ are equivalent we may only consider $\rho$ as a uniform density.
Considering the sequence of transportation maps $\{T_{n}\}_{n\in\N}$, we have % such that $(T_n)_{\#}\nu=\nu_n$. % Section~\ref{mappe}. 
%Then,
\begin{equation}\label{uguaglianza}
\!\!
GF_{n,\delta_{n}}(u_{n})= \frac{2}{\delta_{n}}\iint_{D\times D}\kappa_{\delta_n}( T_{n}(x)-T_{n}(y) )u_{n}( T_{n}(x)) \big(1-u_{n}( T_{n}(y))\big) \rho(x)\rho(y)\,\de x\de y,
\end{equation}
and, by the second condition in \eqref{HP_CPT}, there exists $C>0$ such that
\begin{equation*}
\begin{split}
+\infty>C>&\, GF_{n,\delta_{n}}(u_{n}) \\
=&\, \frac{2}{\delta_{n}^{d+1}}\iint_{D\times D}\eta\bigg(\frac{\vert T_{n}(x)-T_{n}(y) \vert}{\delta_{n}}\bigg)u_{n}( T_{n}(x)) \big(1-u_{n}( T_{n}(y))\big) \rho(x)\rho(y)\,\de x\de y\\
\geq&\, \frac{2a^2}{\delta_{n}^{d+1}}\iint_{D\times D}\eta_0\bigg(\frac{\vert T_{n}(x)-T_{n}(y) \vert}{\delta_{n}}\bigg)u_{n}( T_{n}(x)) \big(1-u_{n}( T_{n}(y))\big) \,\de x\de y\\
\end{split}
\end{equation*}
where we have used, in sequence, the very definition of the rescaled kernel $\kappa_\delta$, hypothesis (H\ref{H3}), and \eqref{eta_0b}.

Observe now that, by the definition of $\lVert\cdot\rVert_\infty$, for almost every $x,y\in D$ the following implication holds true 
\begin{equation}\label{implication}
\vert T_{n}(x)-T_{n}(y)\vert >r_0\delta_{n}\quad\Rightarrow\quad \vert x-y\vert>r_0\delta_{n}-2\norma{{\rm Id}-T_{n}}_{\infty}\,,
\end{equation} 
where $r_0$ has been introduced in \eqref{eta_0a}. 
We define the quantity 
\begin{equation}\label{deltatilde}
\tilde{\delta}_{n}\coloneqq \delta_{n}-2r_0^{-1}\norma{{\rm Id}-T_{n}}_{\infty}
\end{equation}
and notice that, by \eqref{estimaL} and \eqref{delta's}, 
\begin{equation}\label{limite}
\lim_{n\to\infty} \frac{\tilde{\delta}_n}{\delta_n}=1,
\end{equation}
so that, in particular, for large enough $n$, we have that $\tilde{\delta}_n>0$.
Moreover, for almost every $x,y\in D$, \eqref{implication} yields
\begin{align}\label{questoqui}
\eta_0\bigg(\frac{\vert x-y\vert}{\tilde{\delta}_{n}}\bigg)\leq \eta_0\bigg(\frac{\vert T_{n}(x)-T_{n}(y)\vert}{\delta_{n}}\bigg).
\end{align}
By \eqref{limite} and \eqref{questoqui}, we can continue with the chain of inequalities above and obtain that 
\begin{align*}%\label{stimax}
+\infty>C>\frac{1}{\tilde{\delta}_{n}^{d+1}}\iint_{D\times D}\eta_0\bigg(\frac{\vert x-y \vert}{\tilde{\delta}_{n}}\bigg)u_{n}( T_{n}(x)) \big(1-u_{n}( T_{n}(y))\big) \,\de x\de y
\end{align*}
for~$n$ large enough.
Then by \cite[Theorem~3.1]{lussardi2013} $\{u_{n}\circ T_{n}\}_{n\in\N}$ is relatively compact in~$L^{1}_{\nu}(D)$ so that by Remark~\ref{iofareiunremark}, condition \eqref{stagn2} is satisfied with $p=1$.
Therefore, by Proposition~\ref{prop2.7}-5, we obtain that the pair $\{(\nu_n,u_{n})\}_{n\in \N}$ is relatively compact in~$TL^{1}(D)$. 
%
%By Remark~\ref{iofareiunremark}, , so that 
%
%
%\vfill
%
%
%Then by Theorem \ref{terzol} $\{u_{n}\circ T_{n}\}_{n\in\N}$ is relatively compact in~$L^{1}(D;\nu)$ and so the pair $\{(\nu_n,u_{n})\}_{n\in \N}$ is relatively compact in~$TL^{1}(D)$. 
%For general functions~$\eta$ satisfying (H2) it is sufficient to choose a function $\eta_0\leq \eta$ of the previous form. The conclusion follows.
%We now consider some $m\in\N$ and $\eta=\sum_{k=1}^{m}\eta_{k}$ with $(\eta_{k})_{k=1}^{m}$ functions of the previous form. We then apply the same reasoning as before and thus $\eta_{k}$ satisfies \eqref{stimax}, for all $k=1,\ldots,m$. By summing over all $k$, we obtain the same conclusion for $\eta$. A similar argument can be done in the case of a compactly supported kernel or in general, a kernel satisfying our assumption (H1). Let us now suppose that $\rho$ is bounded from below and above by positive constants. 
\end{proof}

In order to prove Theorem~\ref{ourgamma}, we define the auxiliary functional $\F_{\delta}(\cdot;\rho)\colon L^{1}(D)\rightarrow [0,+\infty]$ defined as
\begin{align}\label{glimit1rho}
\!\!
\F_{\delta}(u;\rho)\coloneqq
\begin{cases}
\displaystyle \frac{2}{\delta}\iint_{D\times D}  \kappa_{\delta}(x-y)(1-u(x))u(y)\rho(x)\rho(y) \,\de x\de y & \text{if $0\leq u\leq 1$,}\\[3mm]
+\infty & \text{otherwise in $ L^{1}(D)$.}
\end{cases}
\end{align}
To prove the following result, our inspiration comes from the proof strategy of the main result in \cite{lussardi2013}, which, in turn, relies on the main result of \cite{alberti1998}.
The idea in \cite{lussardi2013} was to notice that their functional $u\mapsto F_\eps(u)$ (see \cite[Section~3]{lussardi2013}) could be written, up to the change of variables $v=2u-1$, as the functional (see \cite[formula (1.1)]{alberti1998}) $v\mapsto F_\eps(v)\eqqcolon \AB_\eps(v;J,W)$ for a suitable choice of the interaction kernel~$J$ and of the potential~$W$. 
\begin{proposition}\label{questaelaproposizionepiuimportante}
The sequence $\{\F_{\delta}(\cdot; \rho)\}_{\delta>0}$ $\Gamma$-converges with respect to the $L^{1}(D)$ metric to $\alpha_{d}TV(\cdot;\rho^{2})$.
\end{proposition}
\begin{proof}
Notice that, upon defining $v\coloneqq 2u-1\in L^1(D)$, we have
\begin{equation}\label{somma}
\F_\delta(u;\rho)=\Phi_\delta (v;\kappa,\rho)+\Psi_\delta(v;\kappa,\rho)\eqqcolon \widetilde{\AB}_\delta(v;\kappa,\rho),
\end{equation}
where $\Phi_\delta(\cdot;\kappa,\rho),\Psi_\delta(\cdot;\kappa,\rho)\colon L^1(D)\to[0,+\infty]$ are defined by
\begin{equation*}%\label{Phi_delta}
\Phi_\delta (v;\kappa,\rho)\coloneqq
\begin{cases}
\displaystyle \frac1{4\delta} \iint_{D\times D} \kappa_\delta(x-y)(v(x)-v(y))^2\rho(x)\rho(y)\,\de x\de y &\text{if $-1\leq v\leq 1$,}\\[3mm]
+\infty & \text{otherwise},
\end{cases}
\end{equation*}
and
\begin{equation*}%\label{Psi_delta}
\Psi_\delta (v;\kappa,\rho)\coloneqq
\begin{cases}
\displaystyle \frac1{2\delta} \iint_{D\times D} \kappa_\delta(x-y)(1-v^2(x))\rho(x)\rho(y)\,\de x\de y &\text{if $-1\leq v\leq 1$,}\\[3mm]
+\infty & \text{otherwise},
\end{cases}
\end{equation*}
respectively.
Equality \eqref{somma} is a matter of some algebraic computations (see \cite[proof of Theorem~3.1]{lussardi2013}).

\smallskip

\emph{Step 1 -- liminf  inequality}. 
For fixed $\kappa$ and $\rho$, and  $v\in L^1(D;[-1,1])$, we consider the set functions $\widetilde{\AB}_\delta(v;\kappa,\rho;\cdot),\Phi_\delta (v;\kappa,\rho;\cdot),\Psi_\delta (v;\kappa,\rho;\cdot)\colon \cA(D)\to[0,\infty]\to[0,+\infty)$ %$\AB_{\delta}(\cdot,\kappa,\rho,A)\colon L^{1}(D)\rightarrow [0,\infty]$ 
defined as
\begin{equation}\label{localized}
\widetilde{\AB}_\delta(v;\kappa,\rho;A)\coloneqq \Phi_{\delta}(v;\kappa,\rho;A)+\Psi_{\delta}(v;\kappa,\rho;A),
\end{equation}
where
%\begin{subequations}\label{id}
\begin{eqnarray*}
A\mapsto \Phi_{\delta}(v;\kappa,\rho;A)&\!\!\!\!\coloneqq &\!\!\!\! \displaystyle\frac{1}{4\delta}\iint_{A\times A}\kappa_{\delta}(x-y)(v(x)-v(y))^{2}\rho(x)\rho(y)\de x \de y, \\%\label{id1} \\
A\mapsto \Psi_{\delta}(v;\kappa,\rho;A)&\!\!\! \coloneqq &\!\!\!  \frac{1}{2\delta}\iint_{A\times A}\kappa_{\delta}(x-y)(1-v^{2}(x))\rho(x)\rho(y)\de x\de y. %\label{id1A}
\end{eqnarray*}
%\end{subequations}
For given $\bar x\in D$ a point, $r>0$ a sufficiently small radius, and $v\colon D\to\R{}$ a function, we define the $r$-rescaling of $v$ at $\bar x$ by
\begin{align*}%\label{id2}
R_{\bar{x},r}v(x)\coloneqq v(\bar{x}+rx)
\end{align*}
for all $x\in r^{-1}(D-\bar x)$.

We start by proving that for every set $A\in\cA(D)$, $r>0$ small enough, and $\bar{x}\in D$ %\spt(\rho)$
we have
\begin{equation}\label{id3}
\widetilde{\AB}_{\delta}(v;\kappa,\rho; \bar{x}+rA)\geq r^{d-1}(\rho(\bar{x})-rc)^{2}{\AB}_{\delta/r}(R_{\bar{x},r}v;\kappa;A)
\end{equation}
for some constant $c>0$ only depending on the Lipschitz constant of $\rho$. % and the set $A$. 
In \eqref{id3}, the functional $\AB_\eps(w;\kappa;A)$ is the localization of the functional defined by
$$\AB_\eps(w;\kappa)\coloneqq\frac1{4\eps} \iint_{D\times D} \kappa_\eps(x-y)(w(x)-w(y))^2\,\de x\de y+\frac{\tau_d}{2\eps} \int_{D} |1-w^2(x)|\,\de x,$$
which is a particular case of the functional $F_\eps$ in \cite{alberti1998}, with interaction kernel $J=\kappa$ and potential $W(s)=\tau_d|1-s^2|/2$, where $\tau_d=\int_{\R{d}} \kappa(z)\,\de z$.

Indeed, by the changes of variables $x'=r^{-1}(x-\bar x)$ and $y'=r^{-1}(y-\bar x)$, we get
\begin{align*}
&\Phi_{\delta}(v;\kappa,\rho;\bar{x}+rA)= \frac{1}{4\delta}\iint_{(\bar x+rA)\times (\bar x+rA)}\kappa_{\delta}(x-y)(v(x)-v(y))^{2}\rho(x)\rho(y)\de x \de y\\
=&\, \frac{r^{2d} }{4\delta}\iint_{A\times A}\kappa_{\delta}(r(x'-y'))(v(\bar{x}+rx')-v(\bar{x}+ry'))^{2}\rho(\bar{x}+rx')\rho(\bar{x}+ry')\,\de x' \de y' \\
=&\, \frac{r^{d-1} }{4\delta/r}\iint_{A\times A}\kappa_{\delta/r}(x'-y')(v(\bar{x}+rx')-v(\bar{x}+ry'))^{2}\rho(\bar{x}+rx')\rho(\bar{x}+ry')\,\de x' \de y' \\
\geq&\, r^{d-1}(\rho(\bar x)-cr)^2 \iint_{A\times A} \kappa_{\delta/r}(x'-y')(v(\bar{x}+rx')-v(\bar{x}+ry'))^{2}\,\de x' \de y' \\
=&\, r^{d-1}(\rho(\bar x)-cr)^2 \Phi_{\delta/r}(R_{\bar x,r}v;\kappa,1;A),
\end{align*}
where we have used the Lipschitz estimate (for $z'=x',y'$) 
$$|\rho(\bar x+rz')-\rho(\bar x)|\leq {\rm Lip}(\rho)r|z'|\leq {\rm Lip}(\rho)\max_{z\in D}|z|\,r\eqqcolon cr.$$ 
Notice that $\rho(\bar x)-cr>0$ for $r>0$ small enough, owing to (H\ref{H3}).
The same reasoning can be carried out for the functional $\Psi_\delta$, so that \eqref{id3} follows.

Let us now take functions $u,u_\delta\in L^1(D;\nu)$ such that $u_\delta\to u$.
Without loss of generality, we can assume that 
\begin{equation}\label{limitato}
\liminf_{\delta\to0} \F_\delta(u_\delta;\rho)<+\infty
\end{equation} 
(otherwise the $\Gamma$-liminf inequality is trivial).
Then $u_\delta(x)\in[0,1]$ for $\nu$-a.e.~$x\in D$, so that the corresponding $v_\delta=2u_\delta-1\in[-1,1]$ for $\nu$-a.e.~$x\in D$.
By \eqref{somma} and (H\ref{H3}), we can write 
$$
\F_\delta(u_\delta;\rho)= \widetilde{\AB}_\delta(v_\delta;\kappa,\rho)\geq a^2{\AB}_\delta(v_\delta;\kappa)
%=&\, a^2\bigg(\frac1{4\delta} \iint_{D\times D} \kappa_\delta(x-y)(v(x)-v(y))^2\,\de x\de y+\frac{\tau_d}{2\delta} \int_{D} (1-v^2(x))\,\de x\bigg)\\
%=&\, a^2 \AB_\delta(v_\delta;\kappa,W),
$$
%where $\tau_d\coloneqq\int_{R^d} \kappa(z)\,\de z$ (see again the proof of \cite[Theorem~3.1]{lussardi2013} for the details on the constant $\tau_d$) and the potential $s\mapsto W(s)$ is given by $W(s)\coloneqq \tau_d|1-s^2|/2$.
Compactness for the functional $\AB_\delta(\cdot;\kappa)$ with respect to the $L^1$-convergence is proved in \cite[Theorem~3.1]{alberti1998}, so that we obtain that $v=2u-1\in \BV(D;\{-1,1\})$, that is $u\in\BV(D;\{0,1\})$.

Let us now consider the energy densities associated with $\widetilde{\AB}_{\delta}(v_\delta;\kappa,\rho;A)$ and (the localization) $\AB_\delta(v_\delta;\kappa,W;A)$ given, for each $x\in A$, by
%\begin{subequations}\label{id5}
\begin{align*}%\label{id5a}
\begin{aligned}
\tilde g_{\delta}(x)\coloneqq &\, \frac{1}{4\delta}\int_{A}\kappa_{\delta}(x-y)(v_\delta(x)-v_\delta(y))^{2}\rho(x)\rho(y)\,\de y \\
&\, + \frac{1}{2\delta}\int_A \kappa_{\delta}(x-y)(1-v_\delta^{2}(x))\rho(x)\rho(y)\,\de y,
%&W(x,u(x),\rho,D)\coloneqq \vert 1-u^{2}(x)\vert \int_{D}\kappa_{\delta}(x-y)\rho(y)dy.
\end{aligned}
\end{align*}
\begin{equation*}%\label{id5b}
%\begin{aligned}
g_{\delta}(x)\coloneqq  \frac{1}{4\delta}\int_{A}\kappa_{\delta}(x-y)(v_\delta(x)-v_\delta(y))^{2}\rho(x)\rho(y)\,\de y + \frac{1}{\delta}W(v_\delta(x)), %\int_D \kappa_{\delta}(x-y)(1-v_\delta^{2}(x))\rho(x)\rho(y)\,\de y,
%&W(x,u(x),\rho,D)\coloneqq \vert 1-u^{2}(x)\vert \int_{D}\kappa_{\delta}(x-y)\rho(y)dy.
%\end{aligned}
\end{equation*}
%\end{subequations}
and the corresponding energy distributions
\begin{align*}%\label{id6}
\tilde \lambda_{\delta}\coloneqq \tilde g_{\delta} \L^{d}\res A\qquad\text{and}\qquad \lambda_{\delta}\coloneqq g_{\delta} \L^{d}\res A.
\end{align*}
By \eqref{limitato}, the total variation of $\tilde\lambda_\delta$ is bounded uniformly with respect to~$\delta$, so that, up to a (not relabeled) subsequence, $\tilde\lambda_\delta\stackrel*\rightharpoonup\tilde\lambda$, for a certain non-negative Radon measure $\tilde\lambda$.

We now take a point $\bar x\in S_u=S_v$ where $\nu\coloneqq \nu_u(\bar x)=\nu_v(\bar x)$ is well defined, and we compute the Radon--Nikod\'ym derivative of~$\tilde\lambda$ with respect to $\cH^{d-1}\res S_v$ at~$\bar x$.
Letting~$Q_\nu$ be a unit cube with two faces perpendicular to~$\nu$, we have, owing to \eqref{id3},
\begin{equation}\label{questaqua}
\begin{split}
\frac{\de\tilde\lambda}{\de\cH^{d-1}\res S_v}(\bar x)=&\, \lim_{r\to0} \frac{\tilde\lambda(\bar x+rQ_\nu)}{r^{d-1}} =\lim_{r\to0}\lim_{\delta\to0} \frac{\tilde\lambda_\delta(\bar x+rQ_\nu)}{r^{d-1}}\\
=&\, \lim_{r\to0}\lim_{\delta\to0} \frac{\widetilde{\AB}_\delta(v_\delta;\kappa,\rho;\bar x+rQ_\nu)}{r^{d-1}}\\
\geq &\,\lim_{r\to0}\lim_{\delta\to0} (\rho(\bar{x})-rc)^{2}{\AB}_{\delta/r}(R_{\bar{x},r}v_\delta;\kappa;Q_\nu).
\end{split}
\end{equation}
Arguing as in the proof of \cite[Lemma~4.3]{alberti1998}, we can use a standard diagonalization argument and find two sequences $r_n\to0$ and $\delta_n\to0$ such that 
$$\underline\delta_n\coloneqq\delta_n/r_n\to0,\qquad \frac{\de \tilde\lambda}{\de\cH^{d-1}\res S_v}(\bar x)=\lim_{n\to\infty} \frac{\tilde\lambda(\bar x+r_nQ_\nu)}{r_n^{d-1}},$$
and, letting $v_{\underline\delta_n}\coloneqq R_{\bar x,r_n}v_{\delta_n}$,
$$\lim_{n\to\infty} v_{\underline\delta_n}=v\quad\text{in $L^1(Q_\nu)$,}\qquad\text{where}\quad v(x)=v_{\bar{x}}(x)\coloneqq
\begin{cases}
+1 & \text{if $\nm{x,\nu_{u}(\bar{x})}>0$,}\\
-1 & \text{if $\nm{x,\nu_{u}(\bar{x})}<0$.}\\
\end{cases}$$ 
so that we can continue in \eqref{questaqua} with
\[\begin{split}
\frac{\de\tilde\lambda}{\de\cH^{d-1}\res S_v}(\bar x)=&\, \lim_{n\to\infty} \frac{\tilde\lambda(\bar x+r_nQ_\nu)}{r_n^{d-1}}\geq \lim_{n\to\infty} (\rho(\bar{x})-r_nc)^{2} \AB_{\underline\delta_n}(v_{\underline\delta_n};\kappa;Q_\nu) \\
=&\, \rho^2(\bar x) \lim_{n\to\infty} \AB_{\underline\delta_n}(v_{\underline\delta_n};\kappa;Q_\nu)\geq \rho^2(\bar x) \alpha_{d},
\end{split}\]
where the last equality follows since $\AB_{\underline\delta_n}$ is bounded and $r_n\to0$ and the last inequality is obtained as in \cite[Lemma~4.3]{alberti1998}, for our choice of $W$ (see \cite[Theorem~3.1]{lussardi2013} for the details).

\smallskip

\emph{Step 2 -- limsup  inequality}. 
By standard approximation results on sets of finite perimeter (see, \emph{e.g.}, \cite[Theorem~1.24]{Giusti}), it is enough to prove the limsup inequality only for polyhedral functions $u\in\BV(D;\{0,1\})$, that is, functions whose jump set $S_u$ is a polyhedral set, \emph{i.e.}, a set whose faces are the essential union of finitely many affine hyperplanes (see \cite[Definition~5.1]{alberti1998} for the details).

As a matter of fact, the proof of \cite[Theorem~5.2]{alberti1998} can be replicated in its entirety, so we will not report it here. 
Two crucial steps are the covering argument and the rescaling property, both of which need care, due to the presence of the weight $\rho$ in the functional $\F_\delta(\cdot;\rho)$.
In particular, the translation invariance of the set functional $\widetilde{\AB}_\delta(v;\kappa,\rho;\cdot)$ defined in \eqref{localized} is not available in our case (here, as before, $v=2u-1$).

In what follows, we show how to adapt the chain of inequalities \cite[equation (5.8)]{alberti1998} to our case.
Owing to the Lipschitz property of $\rho$, estimate \eqref{id3} can be reversed to obtain
\begin{equation}\label{Lipdasopra}
\widetilde{\AB}_{\delta}(w;\kappa,\rho; \bar{x}+\delta A)\leq \delta^{d-1}(\rho(\bar{x})+\delta c)^{2}{\AB}_{1}(R_{\bar{x},\delta}w;\kappa;A),
\end{equation}
for any function~$w$ in the domain of the functional.
Let now consider a polyhedral set $A\subset D$ as in the proof of \cite[Theorem~5.2]{alberti1998} and let $v\coloneqq 2\chi_A-1$ and let $\{v_\delta\}_{\delta>0}$ be a sequence of Lipschitz functions such that
\begin{enumerate}
\item $v_\delta\to v$ in $L^1(D)$ as $\delta\to0$;
\item $\displaystyle \sup_{\delta>0}\int_{D} |\nabla v_\delta(x)|\,\de x<+\infty$;
\item $\AB_1(R_{\bar x,\delta} v_\delta;\kappa;E)\to\alpha_d$ as $\delta\to0$, for every open set $E\subset D$ and for almost every $\bar x\in S_v$.
\end{enumerate}
Notice that conditions (1) and (2) hold true thanks to standard $\BV$-approximation by Lipschitz functions, while (3) is yielded by the $L^1$-continuity of $\AB_1$\,.

%(our $v_\delta$ is their $u_\eps$).
% and a polyhedral set $A\subset D$ as in the proof of \cite[Theorem~5.2]{alberti1998} (our $v_\delta$ is their $u_\eps$).
By covering~$A$ with finitely many rescaled cubes $\{\bar x_i+\delta C\}_{i=1}^{h(\delta)}$, and using the subadditivity of the integral and \eqref{Lipdasopra}, we have
\begin{equation}\label{preciso}
\begin{split}
\widetilde{\AB}_\delta(v_\delta;\kappa,\rho;A)\leq&\, \widetilde{\AB}_{\delta}\bigg(v_\delta;\kappa,\rho; \bigcup_{i=1}^{h(\delta)}(\bar{x}_i+\delta C)\bigg)\leq \sum_{i=1}^{h(\delta)} \widetilde{\AB}_\delta(v_\delta;\kappa,\rho;\bar x_i+\delta C)\\
\leq&\, \sum_{i=1}^{h(\delta)} \delta^{d-1}(\rho(\bar x_i)+\delta c)^2 \AB_1(R_{\bar x_i,\delta}\,v_\delta;\kappa;C)\\
=&\, \alpha_d\int_{A\cap S_u} \rho^2(x)\,\de\cH^{d-1}(x) +o(1),
\end{split}
\end{equation}
so that inequality \cite[formula~(5.6)]{alberti1998} becomes, in our case, by choosing the recovery sequence $u_\delta=(v_\delta+1)/2$ for every $\delta>0$, % (so that $v_\delta=v$ for every $\delta>0$),
$$\limsup_{\delta\to0} \F_\delta(u_\delta;\rho)=\limsup_{\delta\to0} \widetilde{\AB}_\delta(v_\delta;\kappa,\rho;A)\leq \alpha_d \int_{A\cap S_u} \rho^2(x)\,\de\cH^{d-1}(x).$$
Notice that in \eqref{preciso} we have applied (3) with $E=C$, to obtain
\begin{equation*}%\label{itsisotropic}
\lim_{\delta\to0} \AB_1(R_{\bar x_i,\delta}\,v;\kappa;C)=\AB_1(v_{\bar x_i};\kappa;Q_\nu)=\alpha_d.
\end{equation*}
The proposition is proved.
\end{proof}

\begin{remark}\label{facciamounremark}
We point out that, in the proof of the limsup inequality in Proposition~\ref{questaelaproposizionepiuimportante} above, the Lipschitz regularity of the recovery sequence $\{v_\delta\}_{\delta>0}$ has never been used. 
Indeed, one could take the constant recovery sequence $v_\delta=v$ for every $\delta>0$.
We decided to require Lipschitz regularity so that the very same sequence $\{v_\delta\}_{\delta>0}$ introduced in this proof can be used in the proof of Theorem~\ref{ourgamma} below.
%the Lipschitz regularity of the sequence $\{v_\delta\}_{\delta>0}$ 
\end{remark}

\subsection{Proof of Theorem~\ref{ourgamma}}
In this section, we prove the convergence of $GF_{n,\delta_n}$ to $\alpha_d TV(\cdot;\rho^2)$ with respect to the $TL^{1}$ topology. 
% when $\rho=|D|^{-1}$. 
%Notice that, in this case, 
If $u\in L^1_\nu(D)$ and $0\leq u\leq1$ then $GF_{n,\delta_n}(u)$ is given by \eqref{uguaglianza}.
%\begin{align*}
%GF_{n,\delta_{n}}(u) %&=\frac{2}{\delta_{n}}\iint_{D\times D}\kappa_{\delta_{n}}(T_{n}(x)-T_{n}(y)) \big(1-u(T_{n}(x))\big) u(T_{n}(y))\rho(x)\rho(y)\,\de x\de y\\
%=\frac{2}{\delta_{n}\vert D \vert^{2}}\iint_{D\times D}\kappa_{\delta_{n}}(T_{n}(x)-T_{n}(y)) \big(1-u(T_{n}(x))\big) u(T_{n}(y))\,\de x\de y,
%\end{align*}
%where $T_n$ are transport maps such that $(T_{n})_{\#}\nu=\nu_{n}$; moreover, the limiting functional defined in \eqref{TVrho2} reads
%\begin{align}\label{perDrho2}
%TV(u;\rho^{2})=\begin{cases}
%|D|^{-2}\H^{d-1}(J_u) & \begin{array}{l}
%	\text{if $u\in {\rm BV}(D;\rho)$}\\
%	\text{and $u(x)\in\{0,1\}$ for $\nu$-a.e.~$x\in D$}
%	\end{array}\\[3mm]
%+\infty & \;\;\text{otherwise in $L^1(D;\nu)$.}
%\end{cases}
%\end{align}
We divide the proof into two steps.

\emph{Step 1 -- liminf inequality.}
Let $u\in L^1_\nu(D)$ and let $u_{n}\in L^{1}_{\nu_n}(D)$ be such that $(\nu_n,u_n)\stackrel{TL^1}{\longrightarrow}(\nu,u)$ (without loss of generality, we can assume that $0\leq u_n\leq 1$ for every $n\in\N$). 
By Remark~\ref{remark31}(1), let us consider a kernel~$\eta_0$ as in~\eqref{eta_0}.
Recall that for almost every $x,y\in D$ the implication in \eqref{implication} holds true and, defining $\tilde{\delta}_n$ as in \eqref{deltatilde}, we have that
\begin{align*}
GF_{n,\delta_{n}}(u_{n})\geq  \bigg(\frac{\tilde{\delta}_{n}}{\delta_{n}} \bigg)^{d+1}\F_{\delta_{n}}(\tilde{u}_{n};\rho),
\end{align*}
where $\tilde{u}_{n}=u_{n}\circ T_{n}$ and $\F_{\delta_{n}}(\cdot;\rho)$ is defined in \eqref{glimit1rho}.
By \eqref{limite} and Proposition~\ref{questaelaproposizionepiuimportante}, we obtain that $u\in\BV(D;\{0,1\})$ and 
\begin{align*}
\liminf_{n\rightarrow \infty} GF_{n,\delta_{n}}(u_{n})\geq %\frac1{|D|^2} G(u)=\frac{\alpha_d}{|D|^2}\H^{d-1}(J_{u}).
\alpha_d TV(u;\rho^2).
\end{align*}

For general kernel $\eta$ as in (H\ref{H2}), we can find a sequence $\{\eta^m\}_{m\in\N}$ such that each~$\eta^m$ is of the form
$$\eta^m=\sum_{i=1}^m \eta_0^i$$ 
for some $\eta_0^i$ as in~\eqref{eta_0}, for $i=1,\ldots,m$, and $\eta^m\uparrow \eta$.
Set $GF_{n,\delta_{n}}^{m}$ and $GF_{n,\delta_{n}}^{0,i}$ to be the functional  in \eqref{tilde} defined with kernels $\eta^m$ and $\eta^i_0$, respectively (keep hypothesis (H\ref{H2}) into account), and $\alpha_d^{0,i}$ and $\alpha_d^m$ the constants in \eqref{costant} relative to the kernels $\eta^i_0$ and $\eta^m$, respectively (notice that $\alpha_d^m=\sum_{i=1}^m \alpha_d^{0,i}$).
By the previous argument and the superadditivity of the $\liminf$, we can write
\begin{align*}
\liminf_{n\rightarrow \infty} GF_{n,\delta_{n}}(u_{n})\geq \liminf_{n\to\infty} GF_{n,\delta_n}^m(u_n)\geq \sum_{i=1}^{m}\liminf_{n\rightarrow \infty}GF_{n,\delta_{n}}^{0,i}(u_{n}) 
\geq  \sum_{i=1}^{m} %\frac{\alpha_d^{0,i}}{|D|^2}\H^{d-1}(J_{u}). %=\rosso{\frac{2\sigma_d // \alpha_d}{|D|^2}} \H^{d-1}(J_{u}).
\alpha_d^{0,i} TV(u;\rho^2).
\end{align*}
By the monotone convergence theorem, we conclude that 
$$\liminf_{n\rightarrow \infty} GF_{n,\delta_{n}}(u_{n}) \geq \lim_{m\to\infty} \sum_{i=1}^{m} %\frac{\alpha_d^{0,i}}{|D|^2}\H^{d-1}(J_{u}) 
\alpha_d^{0,i} TV(u;\rho^2)= \lim_{m\to\infty} %\frac{\alpha_d^m}{|D|^2} \H^{d-1}(J_{u})
\alpha_d^m TV(u;\rho^2)= %\frac{\alpha_d}{|D|^2} \H^{d-1}(J_{u})
\alpha_d TV(u;\rho^2).$$

\emph{Step 2 -- limsup inequality.}
Observe that it is not restrictive to assume that $u\in \BV(D;\rho)$ with $u(x)\in\{0,1\}$ for $\nu$-a.e.~$x\in D$, and also that $u$ is polyhedral.
Let $u_\delta$ be the recovery sequence for $\F_\delta$ as in the proof of the limsup inequality in  Proposition~\ref{questaelaproposizionepiuimportante}; in particular, the functions $u_{\delta}$ are Lipschitz for every $\delta>0$, $\sup_{\delta>0} \int_D |\nabla u_\delta(x)|\,\de x<+\infty$, and 
$$\limsup_{\delta\to0} \F_{\delta}(u_{\delta};\rho)\leq \alpha_d TV(u;\rho^2).$$
We are going to show that the %constant 
sequence $u_n=u_{\delta_n}*\nu_n$ for all $n\in\N$ is a recovery sequence for~$u$.

By Remark~\ref{remark31}(1), let us consider a kernel~$\eta_0$ as in~\eqref{eta_0}.
Define now 
\begin{equation*}%\label{deltatilde_re}
\tilde{\delta}_{n}\coloneqq \delta_{n}+2 r_0^{-1}\norma{{\rm Id}-T_{n}}_{\infty}
\end{equation*} 
%where $\{T_{n}\}_{n\in\N}$ is a sequence of transportation maps such that $(T_n)_{\#}\nu=\nu_n$, 
and let $\tilde{u}_{n}\coloneqq u_n\circ T_{n}$. Notice that, whereas \eqref{limite} still holds true for this new definition of~$\tilde{\delta}_n$, inequality \eqref{questoqui} is reversed, so that we have that, for almost every $(x,y)\in D\times D$,
\begin{align*}%\label{quellola}
\eta_0\bigg(\frac{\vert x-y \vert}{\tilde{\delta}_{n}}\bigg)\geq \eta_0\bigg(\frac{\vert T_{n}(x)-T_{n}(y) \vert}{\delta_{n}}\bigg).
\end{align*}
Then, for all $n\in\N$,
\begin{align}\label{dis2}
\begin{aligned}
&\frac{1}{\tilde{\delta}_{n}^{d+1}}\iint_{D\times D}\eta_0\bigg(\frac{\vert T_{n}(x)-T_{n}(y)\vert}{\delta_{n}}\bigg)\tilde{u}_{n}(x)(1-\tilde{u}_{n}(y))
%\rho(x)\rho(y)
\,\de x\de y\\
\leq&\frac{1}{\tilde{\delta}_{n}^{d+1}}\iint_{D\times D}\eta_0\bigg(\frac{\vert x-y\vert}{\tilde{\delta}_{n}}\bigg)\tilde{u}_{n}(x)(1-\tilde{u}_{n}(y))
%\rho(x)\rho(y)
\,\de x\de y.
\end{aligned}
\end{align}

Since for $x,y\in D$
\begin{align*}
&\,u_n(x)(1-u_n(y))-\tilde u_n(x)(1-\tilde u_n(y))\\
=&\,(u_n(x)-\tilde u_n(x))(1-u_n(y)) + \tilde u_n(x)(\tilde u_n(y)-u_n(y))
\end{align*}
and $0\leq u_n\leq 1$, recalling (H\ref{H3}), we have
\begin{align*}
&\frac{1}{\tilde{\delta}_{n}}\left\vert \iint_{D\times D} \eta_{0}\bigg(\frac{|x-y|}{\tilde{\delta}_{n}}\bigg)\left(u_n(x)(1-u_n(y))-\tilde u_{n}(x)(1-\tilde u_{n}(y)) \right)\rho(x)\rho(y)\,\de x\de y\right\vert \\
\leq&\,\frac{b^2}{\tilde{\delta}_{n}} \iint_{D\times D} \eta_{0}\bigg(\frac{|x-y|}{\tilde{\delta}_{n}}\bigg) |(u_n(x)-\tilde u_n(x))(1-u_n(y)) + \tilde u_n(x)(\tilde u_n(y)-u_n(y))| \,\de x\de y\\
\leq&\, \frac{2b^2}{\tilde{\delta}_{n}}\iint_{D\times D} \eta_{0}\bigg(\frac{|x-y|}{\tilde{\delta}_{n}}\bigg) |u_n(x) -u_n(T_{n}(x))| \, \de x\de y\\
\leq&\, \frac{C}{\tilde{\delta}_{n}} \int_{D} |u_n(x) -u_n(T_{n}(x))| \, \de x \leq \frac{C\lVert {\rm Id}-T_n\rVert_\infty}{\tilde{\delta}_{n}} \int_{D} |\nabla u_{\delta_n}(x)| \, \de x \to 0
\end{align*}
owing to the boundedness of the last integral and to Remark~\ref{iofareiunremark}.
%where $\tilde{C}\equiv \int_{\R^{d}}\kappa(h)dh$. Since the last term goes to zero as $n\rightarrow \infty$,  one has
%
%\begin{align*}
%\lim_{n\rightarrow \infty}\frac{1}{\tilde{\delta_{n}}}\left\vert \int_{D\times D}\eta_{\tilde{\delta_{n}}}(\vert x-y\vert)\left(u(x)(1-u(y))-u( T_{n}(x))\left(1-u( T_{n}(y))\right) \right)\rho(x)\rho(y)\,\de x\de y\right\vert=0.
%\end{align*}
This implies, invoking \eqref{limite} and inequality \eqref{dis2}, that
\begin{align*}
&\, \limsup_{n\rightarrow \infty} GF_{n,\delta_{n}}(u_{n})\\
=&\,\limsup_{n\rightarrow \infty}\frac{2}{\tilde{\delta}_{n}^{d+1}}\iint_{D\times D} \eta_0\bigg(\frac{\vert T_{n}(x)-T_{n}(y)\vert}{\delta_{n}}\bigg)\tilde u_{n}(x)(1-\tilde u_{n}(y))\rho(x)\rho(y)\,\de x\de y\\
\leq&\,\limsup_{n\rightarrow \infty}\frac{2}{\tilde{\delta}_{n}^{d+1}}\iint_{D\times D} \eta_0\bigg(\frac{| x-y|}{\tilde\delta_n}\bigg) \tilde u_n(x)(1-\tilde u_{n}(y))\rho(x)\rho(y)\,\de x\de y\\
=&\, \limsup_{n\rightarrow \infty}\frac{2}{\tilde{\delta}_{n}^{d+1}}\iint_{D\times D} \eta_0\bigg(\frac{| x-y|}{\tilde\delta_n}\bigg) u_n(x)(1- u_{n}(y))\rho(x)\rho(y)\,\de x\de y\\
=&\,\limsup_{n\rightarrow \infty}\F_{\tilde{\delta}_{n}}({u}_{\tilde{\delta}_n};\rho)\leq \alpha_d TV(u;\rho^2).
\end{align*}
%where $c(\eta,D)$ is defined as in \eqref{cright}. 

The case of general $\eta$ satisfying (H\ref{H2}) can now be obtained by adapting the argument of Step 2 in the proof of \cite[Theorem~4.1]{garcia2016}, which concludes the proof of our theorem. \qed

\subsection{Proofs of Lemma~\ref{corb1} and Theorem~\ref{mainresult}}
%We  start by  giving the argument of the compactness result stated in Theorem \ref{corb1}.
\begin{proof}[Proof of Lemma~\ref{corb1}]
By \eqref{ultima}, we have that $GF_{n,\delta_n}(\chi_{A_n})$ is uniformly bounded with respect to~$n$, so that, by combining Lemma~\ref{compatt} and Proposition~\ref{prop2.7}-5, there exists a set $A\in\B(D)$ with finite perimeter such that $\chi_{A_n}\to\chi_A$ in $L^1(D)$, obtaining the first convergence in \eqref{hypothesis}.
%For every $n\in\N$, let $A_n\in\B(D)$ be such that 
%\begin{equation}\label{succ_minim}
%GF_{n,\delta_n}(\chi_{A_n})\leq \inf_{B\in\B(D)} GF_{n,\delta_n}(\chi_B)+\frac1n.
%\end{equation}
%It follows that $\{A_n\}_{n}$ is compact in $L^1(D)$, so that, 

%We now construct the sequence $\{O_n\}_{n}\subset\B(D)$. 
By compactness, there exists a function $\bar\theta\in L^\infty(D)$ such that $\chi_{O_n}\stackrel{*}\rightharpoonup \bar{\theta}$ in $L^{\infty}(D)$.
We observe that $\bar\theta\in\mathcal{C}(A)$, where, for any $B\in\B(D)$, we define
$$\mathcal{C}(B)\coloneqq\bigg\{\theta\in L^1(D):0\leq \theta\leq 1,\, \int_D \theta \,\de x=1,\, \int_{B}\theta \,\de x=0\bigg\}.$$
Since, for every $\widetilde O\in\B(D)$ such that $\chi_{\widetilde O}\in\mathcal{C}(A_n)$, by \eqref{ultima} we have that 
$$W_p(A_n,O_n)\leq W_p(A_n,\widetilde O)+\frac1n,$$
letting $n\to\infty$, we obtain that 
$$W_p(A,\bar\theta)\leq W_p(A,\widetilde O)\qquad\text{for every $\widetilde O\in\B(D)$ such that $\chi_{\widetilde O}\in\mathcal{C}(A)$.}$$
By a relaxation argument, we deduce that 
$$W_p(A,\bar\theta)\leq W_p(A,\theta)\qquad\text{for every $\theta\in\mathcal{C}(A)$,}$$
so that $\bar\theta$ solves the minimization problem
\begin{equation}\label{min_pb}
    \min\big\{ W_{p}(A,\theta): \theta\in\mathcal{C}(A)\big\}
\end{equation}
By \cite[Theorem 3.10]{buttazzo2020wasserstein}, the solution $\bar\theta$ to problem \eqref{min_pb} is the characteristic function, $\bar\theta=\chi_{O}$, of a certain set $O\in\B(D)$.
The second convergence in \eqref{hypothesis} follows and the proof is concluded.
\end{proof}

We now give the proof of Theorem \ref{mainresult}. 
%$L^{1}(D)$ convergence. This is the content of the following remark.
%\begin{remark}\label{contp}
%Consider $(A_{\delta},O_{\delta}), (A,O)\in \X(D)$ such that $A_{\delta}\rightarrow A$ and $O_{\delta}\rightarrow O$, respectively, with respect to the strong convergence in $L^{1}$ as $\delta\rightarrow 0$. Then $A_{\delta}\rightarrow A$, $O_{\delta}\rightarrow O$,  weakly$^{\ast}$ as $\delta\rightarrow 0$, respectively. 
%\end{remark}

\begin{proof}[Proof of Theorem \ref{mainresult}]

We start by observing that the term $(A,O)\mapsto W_{p}(A,O)$ is a continuous perturbation of $GF_{n,\delta_{n}}$ with respect to the topology in which we compute the $\Gamma$-limit, so that it will only be necessary to prove that 
\begin{equation*}%\label{finale_evviva}
\Big(\Gamma(L^1(D))-\lim_{n\to\infty} GF_{n,\delta_n}\Big)(\cdot)=\frac{\alpha_d}{|D|^2}{\rm Per}(\cdot;D).
\end{equation*}
The liminf inequality is a direct consequence of Theorem~\ref{ourgamma}. We now prove the limsup inequality.  
The limsup inequality follows from a standard approximation result once we prove it for a polyhedral set $A\subset D$.
By Remark~\ref{facciamounremark} and Theorem~\ref{ourgamma}, the constant sequence is a recovery sequence for~$A$.
%limsup inequality is proved by taking the constant sequence $ A_{\delta_{n}}=A$ for all $n\in \N$.  Instead, if $A$ is a set of finite perimeter, we can use the fact each set of finite perimeter can be approximated in variation by a sequence of smooth sets, and thus also by polyhedral sets, and we are done.
\end{proof}

\bigskip

\noindent\textbf{Acknowledgements} 
The authors are members of the GNAMPA group (\emph{Gruppo Nazionale per l'Analisi Matematica, la Probabilit\`{a} e le loro Applicazioni}) of INdAM (\emph{Istituto Nazionale di Alta Matematica ``F.~Severi''}). AMH and MM acknowledge support from the MIUR grant \emph{Dipartimenti di Eccellenza 2018--2022} (E11G18000350001) of the Italian Ministry for University and Research.

\bibliographystyle{siam}
\bibliography{cub}
 \end{document}